\documentclass[a4paper,11pt]{article}
\usepackage{a4}
\usepackage{amsmath}
\usepackage{color}
\usepackage{amsfonts}
\usepackage{amssymb}
\usepackage{mathtools}

\DeclarePairedDelimiter\floor{\lfloor}{\rfloor}
\usepackage[square, numbers, comma, sort&compress]{natbib}
\usepackage{amsthm,enumerate}
\usepackage{graphicx}
\usepackage[colorlinks=true,citecolor=black,linkcolor=black,urlcolor=blue]{hyperref}
\def\ds{\displaystyle}
\def\ni{\noindent}
 \makeatletter
 \newtheorem{theorem}{Theorem}[section]

\newtheorem{lemma}[theorem]{Lemma}

\newtheorem{dfn}[theorem]{Definition}

\newtheorem{rmk}[theorem]{Remark}

\def\ni{\noindent}
 \numberwithin{equation}{section}

\begin{document}
\begin{center} \textbf{ \Large{Signless Laplacian spectral characterization of some graphs }
}  \end{center}
\begin{center}
	\textbf{Rakshith B. R.}\vskip0.5cm
	Department of Mathematics\\ Vidyavardhaka College of Engineering,
	 Mysuru - 570 002, INDIA\\
	E-mail:ranmsc08@yahoo.co.in\\
\end{center}
\begin{abstract}
	In this paper, we investigate (signless) Laplacian spectral characterization of graphs with star components. Also, we prove that the join graph ${\small K_{n-\alpha}-e\vee \alpha K_{1}}$ is $DQS$ for $n-\alpha >3$ and $\alpha\neq3$, and for $\alpha=3$, we show that its only $Q$-cospectral mate is $\overline{K_{1,3}\cup K_{2}\cup (n-6)K_{1}}$.   	
\end{abstract}
\ni  \textbf{Keywords}: Laplacian spectrum, signless Laplacian spectrum, cospectral graphs, spectral characterization. \\\\
\ni \textbf{MSC (2010)}:
05C50.
\section{Introduction}
Graphs considered in this paper are simple and undirected. For a graph $G$ with vertex set $V(G)$ (where $|V(G)|=n$) and edge set $E(G)$, the degree of a vertex $v_{i}$ of $G$  is the number of neighbors of $v_{i}$ and is denoted by $d_{i}(G)$. Throughout the paper, we assume that the sequence $\{d_{i}(G)\}_{i=1}^{n}$ is non increasing, i.e., $d_{i}(G)\ge d_{i+1}(G)$, $i=1,2,\ldots,n-1$. Let $A(G)$ be the adjacency matrix of $G$ and $D_{_G}$ be the diagonal matrix of degrees of $G$. The Laplacian matrix $L(G)$ and signless Laplacian matrix $Q(G)$ of the graph $G$ are defined as $L(G)=D_G-A(G)$ and $Q(G)=D_G+A(G)$, respectively.  We denote the Laplacian eigenvalues of  $G$  by $\mu_{1}(G)\ge \mu_{2}(G)\ge \ldots \ge \mu_{n}(G)$ and let $\gamma_{1}(G)\ge \gamma_{2}(G)\ge \ldots \ge \gamma_{n}(G)$ be the signless Laplacian eigenvalues of $G$.  The Spectrum of the matrix  $A(G)$ (resp.
$L(G)$ and $Q(G)$) is called the $A$-spectrum (resp. $L$-spectrum and $Q$-spectrum) of $G$. We denote by $Spec_{L}(G)$, the Laplacian spectrum of $G$ and $Spec_{Q}(G)$ denote the signless Laplacian spectrum of $G$.  It may be noted that for a bipartite graph, its Laplacian spectrum and signless Laplacian spectrum coincides. Two graphs are said to be  $A$-cospectral (resp. $L$-cospectral and $Q$-cospectral) if they have the same $A$-spectrum (resp. $L$-spectrum, $Q$-spectrum). If two graphs are $A$-cospectral (resp. $L$- cospectral, $Q$-cospectral), then they are not necessarily isomorphic. A graph $G$ is said to be $DS$ (resp. $DLS$, $DQS$) if there is
no other non-isomorphic graph $A$-cospectral (resp. $L$-cospectral, $Q$-cospectral) with $G$. Interestingly, a graph $G$ is $DLS$ if and only its complement $\overline{G}$ is $DLS$. The problem of characterizing graphs which are determined by their spectra is one of the interesting and difficult problem in spectral graph theory. This problem was posed by G\"{u}nthard and Primas \cite{primas} with motivations from H\"{u}ckel theory. In \cite{Dam,Dam1}, Dam and Haemers gave a survey of answers to the question of which graphs are determined by the spectrum of some matrix associated to the graph.   In literature, there are several papers addressing the problem of characterizing graphs which are $DS$, $DLS$ and $DQS$. In fact, only few graphs with special structures have been proved to be determined by their spectra. For some recent papers on this topic, see \cite{kite,R3,dsi,butter,R2}.
As usual, $P_{n}$, $C_{n}$, $S_{n}$ and $K_{n}$ stand for the path, the cycle, the star graph and the complete
graph of order $n$, respectively. We use $K_{n,m}$ to denote the complete bipartite graph with vertex partition sets of order $n$ and $m$, and let $G-e$ be the graph obtained from $G$ by deleting the edge $e$ of $G$.\\\\ Recently, (signless) Laplacian spectral characterization of disjoint union of graphs have been studied and also the problem of characterizing join graphs which are determined by their (signless) Laplacian  spectra is considered.   In \cite{LJ, LJ1, LJ2, dsi}, (signless) Laplacian  characterization of graphs with isolated vertices were studied. In \cite{cou}, Abdian et al. investigated signless Laplacian spectral characterization of graphs with independent egdes and isolated vertices. In \cite{j2,j3,j4,j5,split}, the join graph $G\vee K_{r}$ is proved to be $DQS$ under certain conditions on $G$, namely, when $G$ is a  $k$-regular graph for $k=1,2, n-2,n-3$ and if $G$ is a complete split graph. Motivated by these works in Section 2 of this paper, we investigate (signless) Laplacian spectral characterization of graphs with star components and in    Section 3, we prove that the join graph ${\small K_{n-\alpha}-e\vee \alpha K_{1}}$ is $DQS$ for $n-\alpha >3$ and $\alpha\neq3$, and for $\alpha=3$, we show that its only $Q$-cospectral mate is $\overline{K_{1,3}\cup K_{2}\cup (n-6)K_{1}}$.

\section{Some disjoint union of graphs determined by their signless Laplacian spectra}\label{section}
In this section, we first give a list of lemmas that are useful to prove our main results. We denote by $G_{p,n,q}$, the double starlike tree obtained by attaching $p$ pendant vertices to
one pendant vertex of the path $P_n$ and $q$ pendant vertices to the other pendant vertex of $P_n$. Let $m^{-}_{_G}(\lambda)$ ( resp.  $m^{+}_{_G}(\lambda)$ ) be the multiplicity of the Laplacian (resp. signless Laplacian) eigenvalue $\lambda$ of $G$. 
\begin{lemma}[\cite{bookc}]\label{l1}
	Let $G$ be a graph with $n$ vertices and $m$ edges. Then $\gamma_{1}(G) \ge \dfrac{4m}{n}$, with equality if and only if $G$ is
	regular graph of degree $\dfrac{\gamma_{1}(G)}{2}$.
\end{lemma}
\begin{lemma}[\cite{bookc}]\label{l2}
For a connected graph $G$, we have $\mu_{1}(G)\le n$.
\end{lemma}
\begin{lemma}[\cite{a2}]\label{lq}
	If G is a bipartite graph, then the Q-spectrum of G is equal to its L-spectrum.
\end{lemma}
\begin{lemma}[\cite{mull}]\label{mull}
Let $T$ be a tree. Then we have\\
$(1)$ $m^{-}_{T}(1)=n-4$ if and only if $T$ is a double star graph.\\
$(2)$ $m^{-}_{T}(1)=n-5$ if and only if $T\cong G_{p,3,q}$ or $T\cong G_{p,4,q}$.
\end{lemma}
\begin{lemma}[\cite{p1}]\label{p}
If H is a proper subgraph of a connected graph G, then $\gamma_{1}(H) < \gamma_{1}(G)$.
\end{lemma}
Let $p^{-}(G)$ (resp. $p^{+}(G)$) denote the product of non-zero Laplacian (resp. singless Laplacian) eigenvalues of $G$.
\begin{lemma}[\cite{bookc}]\label{pro}
For any connected bipartite graph $G$ of order $n$, we have $p^{+}(G) = p^{-}(G) = n\tau (G)$, where $\tau (G)$ is the
number of spanning trees of $G$.
\end{lemma}
\begin{lemma}[\cite{d4}]\label{d4}
For any graph $G$, $det(Q(G)) = 4$ if and only if $G$ is an odd unicyclic graph.
\end{lemma}

\begin{lemma}[\cite{dh}]\label{dh}
The  graph $G_{p,n,q}$ is determined by its Laplacian spectrum.
\end{lemma}
\begin{lemma}[\cite{dsi}]\label{dsi}
Let $T$ be a DLS tree of order $n$. Then $T\cup rK_1$ is DLS. 
\end{lemma}
\begin{lemma}[\cite{Dam}]\label{dlc}
 From the adjacency spectrum, the Laplacian spectrum and from the signless Laplacian spectrum of $G$, the following invariants of $G$ can be obtained.\\
$(1)$ The number of vertices.\\
$(2)$ The number of edges.\\
From the Laplacian spectrum, the following invariant of $G$ can be deduced.\\
$(3)$ The number of components.\\
From the signless Laplacian spectrum, the following invariant of $G$ can be obtained.\\ 
$(4)$ The number of bipartite components.
\end{lemma}
\begin{lemma}[\cite{dsla}]\label{dlsa}
	Let $G$ be a connected graph with $n \ge 3$ vertices.  Then
	$\mu_{2}(G) \ge d_{2}(G)$.
	\end{lemma}
\begin{lemma}[\cite{algebraic}]\label{dl1}
Let $T$ be a tree of order $n$. Then
  $\mu_{n-1}(T)\le 1$ and the equality holds if and only if $T\cong S_{n}$.
\end{lemma}
\begin{lemma}[\cite{lst,siam1}]\label{tree}
Let $T$ be a tree of order $n$. Then  $\mu_{1}(T)=n$ if and only if $T\cong S_{n}$.	
\end{lemma}
\begin{lemma}[\cite{RBBbook},\cite{siam}]\label{siam}
	Suppose $T$ is a tree on n vertices. If $\mu>1$ is an integer Laplacian eigenvalue of $T$, then $\mu|n$ and $m^{-}_{T}(\mu)=1$.
	\end{lemma}
\begin{lemma}[\cite{kcdas}]\label{kcdas}
Let $G$ be a connected graph of order $n$. Then
$\gamma_{n}(G) < d_{n}$.	
\end{lemma} 
\begin{lemma}[\cite{U2}]\label{U2}
Let $U$ be an unicyclic graph. If $\lambda>1$ is an integer signless Laplacian eigenvalue of $U$, then $m^{+}_{U}(\lambda)\le 2$. 
\end{lemma}
\begin{theorem}[\cite{mul}]\label{mul}
Let $G$ be a connected graph of order $n\ge 4$. Then $G$ has a $Q$-eigenvalue $c$ of
multiplicity $n -2$ if and only if $G$ is one of the graphs $K_n - e$, $S_n$, $K_{n/2,n/2}$, $\overline{K_3\cup S_4}$ and $\overline{K_1 \cup 2K_3}$.
\end{theorem}
The following lemma gives the signless Laplacian spectrum of the double star graph $G(p, 2, q)$. 
\begin{lemma}\label{dls}
The signless Laplacian spectrum of the double star graph $G(p,2,q)$ consists of
\begin{enumerate}[a.]
\item	
1 with multiplicity $p+q-2$;
\item	
0 with multiplicity 1;
\item
the three roots of the polynomial $x^3+(-p-q-4)x^2+(pq+2p+2q+5)x-p-q-2$.
\end{enumerate}   	
\end{lemma}
\begin{rmk}\label{dr1}
Let $p\ge q$. Then from the above lemma, it can be seen that $\gamma_{1}(G_{p,2,q})\in (p+2,p+3),\gamma_{2}(G_{p,2,q})\in [q+1,q+2)$ and $\gamma_{2}(G_{p,2,q})= q+1$ if and only if $p=q$.	
\end{rmk}

\begin{theorem}\label{t1}
The graph $G_{p,2,q}\cup rS_{m}\cup  s K_{1}$ with $m\ge 2$  is $DLS$.	
\end{theorem}
\begin{proof}
	Let $H$ be a graph $L$-cospectral with $G_{p,2,q}\cup r S_{m}\cup s K_{1}$. By Lemma \ref{dlc} (3), $H$ has exactly $r+s+1$ components and so $H$ has at least $r(m-1)+p+q+1$ edges. Let $H_{1}, H_{2}, \ldots, H_{r+s+1}$ be the components of $H$. Since $G_{p,2,q}\cup r S_{m}\cup s K_{1}$ has   $r(m-1)+p+q+1$ edges by Lemma \ref{dlc} (2), it follows that $H_{1}, H_{2}, \ldots, H_{r+s+1}$ are all trees. Note that  $G_{p,2,q}\cup r S_{m}\cup s K_{1}$ has exactly one non-zero Laplacian eigenvalue strictly less than $1$ and it is $\gamma_{p+q+1}(G_{p,2,q})$. Thus by Lemma \ref{dl1}, we can assume that $H_{1}$ is a tree with $\gamma_{p+q+1}(G_{p,2,q})\in Spec_{_L}(H_{1})$, and all other components of $H$ are star graphs. Now from Lemma \ref{dls} and since $\gamma_{1}(G_{p,2,q})$ is never an integer by Remark \ref{dr1}, we have the following choices for $Spec_{_L}(H_{1})$.  \\
	$Spec_{_L}(H_{1})=\left\{ \begin{array}{c}
		\{\gamma_{1}(G_{p,2,q}),1,1,\ldots,1,\gamma_{p+q+1}(G_{p,2,q}),0\}\\\\
		\{\gamma_{1}(G_{p,2,q}),m,1,1,\ldots,1,\gamma_{p+q+1}(G_{p,2,q}),0\}\\\\	
		\{\gamma_{1}(G_{p,2,q}),\gamma_{2}(G_{p,2,q}),1,1,\ldots,1,\gamma_{p+q+1}(G_{p,2,q}),0\}\\\\
		\{\gamma_{1}(G_{p,2,q}),\gamma_{2}(G_{p,2,q}),m,,1,1,\ldots,1,\gamma_{p+q+1}(G_{p,2,q}),0\}.
		\end{array}\right.$ \\\\ Case I: If $Spec_{_L}(H_{1})=
		\{\gamma_{1}(G_{p,2,q}),1,1,\ldots,1,\gamma_{p+q+1}(G_{p,2,q}),0\}$. Then by Lemma \ref{dlsa}, $d_{2}(H_{1})=1$ and so $H_{1}$ is a star graph. This is impossible, since $\gamma_{p+q+1}(G_{p,2,q})<1$ and $\gamma_{p+q+1}(G_{p,2,q})\in Spec_{L}(H_{1})$.\\\\
		Case II:  $Spec_{_L}(H_{1})=\{\gamma_{1}(G_{p,2,q}),m,1,1,\ldots,1,\gamma_{p+q+1}(G_{p,2,q}),0\}$.\\\\
		Subcase I: Suppose $m>\gamma_{1}(G_{p,2,q})$. From Remark \ref{dr1}, the largest Laplacian eigenvalue of a double star graph is never an integer. Now, since $m^{-}_{H_{1}}(1)= |V(H_{1})|-4$, by Lemma \ref{mull}, it follows that $H_{1}$ must be a double star graph. Thus $H_{1}$ is a double star graph with integer $m$ as its largest Laplacian eigenvalue, a contradiction.\\\\
		Subcase II: Suppose  $m<\gamma_{1}(G_{p,2,q})$. Since $\gamma_{2}(G_{p,2,q})\notin Spec(H_{1})$, $\gamma_{2}(G_{p,2,q})$ must be the largest Laplacian eigenvalue of a star graph  and so $\gamma_{2}(G_{p,2,q})$ must be an integer. Hence by Remark \ref{dr1}, $\gamma_{2}(G_{p,2,q})=q+1$, $p=q$ and    $\floor{\gamma_{1}(G_{p,2,q})}=q+2$. Now, since $m^{-}_{H_{1}}(1)= |V(H_{1})|-4$ by Lemma \ref{mull}, it follows that $H_{1}$ must be a double star graph, i.e.,  $H_{1}\cong G_{p^{\prime},2,q^{\prime}}$, and also since $Spec_{_L}(H_{1})=\{\gamma_{1}(G_{p,2,q}),m,1,1,\ldots,1,\gamma_{p+q+1}(G_{p,2,q}),0\}$, we must have $q^{\prime}=m-1$, $p^{\prime}=q^{\prime}$ and $\floor{\gamma_{1}(G_{p,2,q})}=q^{\prime}+2$, by Remark \ref{dr1}. Thus $\floor{\gamma_{1}(G_{p,2,q})}=q^{\prime}+2=q+2$ and so $q=q^{\prime}$ and $p=p^{\prime}$. Hence $H_{1}\cong G_{p,2,q}$.\\\\
		Case III: If  $Spec_{_L}(H_{1})=\{\gamma_{1}(G_{p,2,q}),\gamma_{2}(G_{p,2,q}),1,1,\ldots,1,\gamma_{p+q+1}(G_{p,2,q}),0\}$, then $H_{1}$ and $G_{p,2,q}$ are $L$-cospectral, and also without loss of generality, we can assume that $H_{2}=H_{3}=\ldots=H_{r+1}=S_{m}$ and $H_{r+2}=\ldots=H_{s+r+1}=K_{1}$. Since $G_{p,2,q}$ is $DLS$ (see Lemma \ref{dh}), it follows that $H_{1}\cong G_{p,2,q}$ and so $H\cong G_{p,2,q}\cup r S_{m}\cup s K_{1}$.\\\\
		Case IV: If $Spec_{_L}(H_{1})=\{\gamma_{1}(G_{p,2,q}),\gamma_{2}(G_{p,2,q}),m,,1,1,\ldots,1,\gamma_{p+q+1}(G_{p,2,q}),0\}$, then the graphs $H_{1}\cup K_{1}$ and $G_{p,2,q}\cup S_{m}$ are $L$-copectral. Since $m^{-}_{H_{1}}(1)=|V(H_{1})|-5$, by Lemma \ref{mull}, $H_{1}\cong G_{1}$ where $G_{1}=G_{p_{1},3,q_{1}}$ or $G_{p_{2},4,q_{2}}$. Thus  $G_{1}\cup K_{1}$ and $G_{p,2,q}\cup S_{m}$ are $L$-copectral. This is impossible, because the graph $G_{1}$ is $DLS$, by Lemma \ref{dh} and so $G_{1}\cup K_{1}$ is $DLS$, by Lemma \ref{dsi}. This completes the proof.         
\end{proof}	

Using the fact that a graph $G$ is $DLS$ if and only if $\overline{G}$ is $DLS$, we obtain the following result from the above theorem.
\begin{theorem}
	The graph $\overline{G_{p,2,q}\cup rS_{m}\cup  s K_{1}}$ with $m\ge 2$  is $DLS$.	
\end{theorem} 
\begin{theorem}
Let $p+q$ and $m$ be odd positive integers. Then the  graph $G_{p,2,q}\cup rS_{m}\cup  s K_{1}$  is $DQS$.	
\end{theorem}
\begin{proof}
	Let $H$ be a graph $Q$-cospectral with $G_{p,2,q}\cup r S_{m}\cup s K_{1}$. By Lemma \ref{dlc} (4), $H$ has exactly $r+s+1$ bipartite components. So $H$ has at least $p+q+r(m-1)+1$ edges. Let $H_{1}, H_{2}, \ldots, H_{r+s+1}$ be the bipartite components of $H$ with $n_{1}, n_{2}, \ldots, n_{r+s+1}$ vertices, respectively. Since $G_{p,2,q}\cup r S_{m}\cup s K_{1}$ has   $p+q+r(m-1)+1$ edges, by Lemma \ref{dlc} (2), it follows that $H_{1}, H_{2}, \ldots, H_{r+s+1}$ are trees and all other components of $H$, if any,  are odd unicyclic graphs. Suppose $H\cong H_{1}\cup H_{2}\cup \ldots \cup H_{r+s+1}$. Then $H$ is $L$-cospectral with $G_{p,2,q}\cup r S_{m}\cup s K_{1}$ and so by Theorem \ref{t1}, $H\cong G_{p,2,q}\cup rS_{m}\cup  s K_{1}$. Suppose $H$ has $c$ $(\ge 1)$ number of odd unicyclic components, then by Lemmas \ref{pro} and \ref{d4}, $4^{c}n_{1}n_{2}\ldots n_{r+s+1}=(p+q+2)m^{r}$. This is impossible because $(p+q+2)m^{r}$ is odd. This completes the proof.
\end{proof}	
\begin{theorem}
Let $p>1$ be an odd integer and let $G$ be an odd unicyclic $DQS$ graph of order $n$ such that $m^{+}_{G}(p)=2$. Then the graph $G\cup r S_{p}\cup s K_{1}$ is $DQS$.
\begin{proof}
Let $H$ be a graph $Q$-cospectral with $G\cup r S_{p}\cup s K_{1}$. By Lemma \ref{dlc} (4), $H$ has exactly $r+s$ bipartite components and so $H$ has at least $n+r(p-1)$ edges. Let $H_{1}, H_{2}, \ldots, H_{r+s}$ be the bipartite components of $H$ with $n_{1}, n_{2}, \ldots, n_{r+s}$ vertices, respectively. Since $G\cup r S_{p}\cup s K_{1}$ has   $n+r(p-1)$ edges and $H$ has at least $n+r(p-1)$ edges, the components $H_{1}, H_{2}, \ldots, H_{r+s}$ are trees and all other components of $H$, if any,  are odd unicyclic graphs, by Lemma \ref{dlc} (2). Suppose $H\cong H_{1}\cup H_{2}\cup \ldots \cup H_{r+s}$. Since $m^{+}_{H}(p)=r+2$, by Lemma \ref{siam}, $r+2$ components of $H$ say $H_{1}, H_{2}, H_{3}, \ldots, H_{r+2}$ have $p$ as one of its signless Laplacian eigenvalue and also $p$ must divide $n_{i}$ for $1\le i\le r+2$, i.e., $n_{i}=pk_{i}$, for some integer $k_{i}$. Therefore,  from Lemmas \ref{pro} and \ref{d4}, we have $4p^r=p^{r+2}k_{1}k_{2}\ldots k_{r+2}n_{r+3}\ldots n_{r+s}$. This implies that $p| 4$, a contradiction. Hence $H$ has $c$ $(\ge1)$ number of odd unicyclic graphs. Now, again by Lemmas \ref{d4} and \ref{pro},  $4p^r=4^cn_{1}n_{2}\ldots n_{r+s}$ and so $c=1$, since $p$ is odd. Let $C$ be an odd unicyclic component of $H$. From Lemmas \ref{U2} and \ref{siam}, at least $r$ bipartite components (trees) of $H$, say $H_{1}, H_{2}, \ldots H_{r}$ have $p$ as one of its signless Laplacian eigenvalue and also $p$ must divide $n_{i}$ for $1\le i\le r$, i.e., $n_{i}=pk_{i}$, for some integer $k_{i}$. Hence $p^{r}=p^rk_{1}k_{2}\ldots k_{r}n_{r+1}n_{r+2}\ldots n_{r+s}$ and so $k_{1}=k_{2}=\ldots=n_{r+1}=n_{r+2}=\ldots=n_{r+s}=1$. Thus $H_{i}$, $1\le i\le r$ is a tree with $p$ vertices and $H_{i}=K_{1}$, $i=r+1,r+2,\ldots,r+s$. Now, since $p$ is a signless Laplacian eigenvalue of $H_{i}$ ($1\le i\le r$), a tree with $p$ vertices,  by Lemma \ref{tree}, $H_{i}=S_{p}$. Thus $C\cong G$ and $H \cong G \cup r S_{p}\cup sK_{1}$.   	
\end{proof} 	
\end{theorem}
\begin{theorem}
Let $G$ be an odd unicyclic $DQS$ graph on $n$ vertices such that $\gamma_{n-1}(G)\ge1$ and let $p>1$ be an odd integer such that no proper divisors of $p$ is a signless Laplacian eigenvalue of $G$ and $n\ge 3p+1$. Then the graph $G\cup  S_{p}\cup sK_{1}$ is $DQS$.
\end{theorem}
\begin{proof}
Let $H$ be a graph $Q$-cospectral with $G\cup  S_{p}\cup s K_{1}$. By Lemma \ref{dlc} (4), $H$ has exactly $s+1$ bipartite components and so $H$ has at least $n+p-1$ edges. Let $H_{1}, H_{2}, \ldots, H_{s+1}$ be the bipartite components of $H$ with $n_{1}, n_{2}, \ldots, n_{s+1}$ vertices, respectively. Since $G\cup  S_{p}\cup s K_{1}$ has   $n+p-1$ edges and $H$ has at least $n+p-1$ edges, it follows that $H_{1}, H_{2}, \ldots, H_{s+1}$ are trees and all other components of $H$, if any,  are odd unicyclic graphs, by Lemma \ref{dlc} (2). Suppose $H\cong H_{1}\cup H_{2}\cup \ldots \cup H_{s+1}$ and $p\in Spec(H_{1})$. Then by Lemma \ref{siam}, $n_{1}=p k_{1}$, for some integer $k_{1}$. Using  Lemmas \ref{d4} and \ref{pro}, we get, $4=k_{1}n_{2}\ldots n_{s+1}$ and so $k_{1}=4, 2 ~\text{or}~ 1$. Thus $\ds\sum_{i=1}^{s+1}{n_{i}}=n+p+s\le s+4p$. This implies  $n\le 3p$, a contradiction. Hence $H$ has $c$ $(\ge1)$ number of odd unicyclic graphs. Now again from Lemmas \ref{d4} and \ref{pro},  $4p=4^cn_{1}n_{2}\ldots n_{s+1}$ and so $c=1$ and $p=n_{1}n_{2}\ldots n_{s+1}$, since $p$ is odd. Let $H_{0}$ be the odd unicyclic component of $H$ of order $n_{0}$. If $\gamma_{n_{0}}(H_{0})<1$, then by hypothesis and  from Lemma \ref{dl1}, 
all the trees $H_{1}, H_{2},\ldots, H_{s+1}$ are star graphs. Since no proper divisors of $p$ is a signless Laplacian eigenvalue of $G$ and also since $p=n_{1}n_{2}\ldots n_{s+1}$, we have, $H_{1}= S_{p}$ and $H_{2}= H_{3}=\ldots =H_{s+1}= K_{1}$. Hence $H_{0}\cong G$ and $H\cong G\cup  S_{p}\cup sK_{1}$. If $\gamma_{n_{0}}(H_{0})\ge1$, then by Lemma \ref{kcdas}, $H_{0}$ must be a cycle. Since $\gamma_{m}(C_{m})<1$, for $m\ge 5$,  $H_{0}$ must be $C_{3}$. As $p=n_{1}n_{2}\ldots n_{s+1}$, we have $\ds\sum_{i=1}^{s+1}n_{i}\le p+s$ and since $n+p+s=3+\ds\sum_{i=1}^{s+1}n_{i}$, it follows that $n\le3$, a contradiction. This completes the proof.      
\end{proof}
The proof of our next theorem is similar to that of the above theorem.
\begin{theorem}
Let $G$ be an odd unicyclic $DQS$ graph on $n\ge10$ vertices. Then the graph $G\cup S_{3}\cup s K_{1}$ is $DQS$.	
\end{theorem}
In \cite{LJ}, Sun et al. showed that if $G$ is a ($n,$ $m$) graph with $n\ge 10$ and $m\ge n(n-3)/2+4$, then the graph $G\cup K_{2}\cup rK_{1}$ is $DLS$. Motivated by this result, we obtain the following theorem. 
\begin{theorem}
Let $G$ be a connected $(n, m)$ graph with $\small{m\ge \dfrac{(n-2)(n-3)}{2}+5}$. Suppose $H$ is a graph Q-cospectral with $G\cup K_{2}\cup rK_{1}$. Then \begin{enumerate} [a.]
	\item $H\cong H_{1}\cup (r+1) K_{1}$;
	\item
	$H\cong H_{1}\cup K_{2}\cup r K_{1}$;
	\item
	$H\cong H_{1}\cup K_{2}\cup K_{2}\cup (r-1) K_{1}$;
	\item
	$H\cong H_{1}\cup K_{1,2}\cup r K_{1}$, where $H_{1}$ is a connected graph.
\end{enumerate}     
\end{theorem}	
\begin{proof}
Since $G\cup K_{2}\cup rK_{1}$ and $H$ are $Q$-cospectral, by Lemma \ref{dlc} (4), $H$ has at least $r+1$ components. Let us assume that $H$ has exactly $r+1$ components. Suppose $H_{1}$ is a bipartite component of $H$ with $n_{1}$ vertices and $\gamma_{1}(H_{1})=\gamma_{1}(G)$. Then from Lemmas \ref{l1} and \ref{l2}, we have
\begin{eqnarray}\label{eq1}
	n+2\ge n_{1}\ge\gamma_{1}(H_{1})=\gamma_{1}(G)\ge \dfrac{4m}{n}\ge \dfrac{2(n-2)(n-3)+20}{n}.
\end{eqnarray}		
From (\ref{eq1}), it follows that $4\le n\le 8$. \\\\
Case I: $n=4$ or $n=8$. If $n=4$, then by (\ref{eq1}), $H_{1}$ is a bipartite graph with 6 vertices and 7 edges. Also, we must have $\gamma_{1}(H_{1})=6$. This is impossible, because $H_{1}$ is  a proper subgraph of a complete bipartite graph  with 6 vertices and so by Lemma \ref{p}, $\gamma_{1}(H_{1})<6$. If $n=8$, then $H_{1}$ is a bipartite graph with 10 vertices and 21 edges. Also, $2\in Spec_{Q}(H_{1})$. This is impossible.\\
 \begin{figure}
	\begin{align*}
	\includegraphics[height=0.85in]{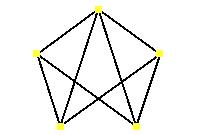}&~~~
	\includegraphics[height=0.85in]{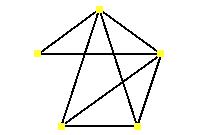}
	\end{align*}
	\caption{Non isomorphic graphs with 5 vertices and 8 edges.}
	\label{fig1}
\end{figure}

\ni Case II: If $n=5$, then by (\ref{eq1}), $H_{1}$ is a bipartite graph with 7 vertices and 9 edges, and $G$ is a non bipartite graph with 5 vertices and 8 edges. Also, $\gamma_{1}(H_{1})\ge\dfrac{32}{5} $. By using Maple, we see that  $K_{2,5}-e$ is the only bipartite graph with 7 vertices and 9 edges such that its largest singless Laplacian eigenvalue greater than 6.4. Also, there are 2 non isomorphic graphs with 5 vertices and 8 edges, see Figure \ref{fig1}. But the largest signless Laplacian eigenvalue of these two graphs are different from that of $K_{2,5}-e$.\\\\
Case III: If $n=6$, then by (\ref{eq1}), $H_{1}$ is a bipartite graph with 8 vertices and 12 edges, and $G$ is a non bipartite graph with 6 vertices and 11 edges. By using Maple, we see that $K_{2,6}$ is the only bipartite graph with 8 vertices and 12 edges having its largest signless Laplacian eigenvalue greater than $\dfrac{22}{3}$. Also, there are 9 non isomorphic graphs with 6 vertices and 11 edges, see Fig. \ref{fig2}. But the second largest signless Laplacian eigenvalue of these 9 graphs are not equal to $\gamma_{2}(K_{2,6})=6$. 
\begin{figure}[h]
	\begin{align*}
	\includegraphics[height=0.85in]{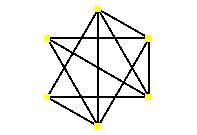}&~~~
	\includegraphics[height=0.85in]{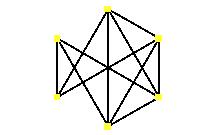}~~~~
	\includegraphics[height=0.85in]{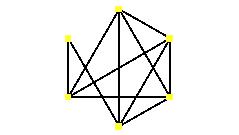}\\
	\includegraphics[height=0.85in]{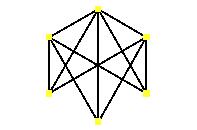}&~~~
	\includegraphics[height=0.85in]{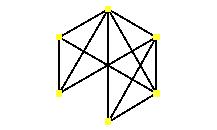}~~~
	\includegraphics[height=0.85in]{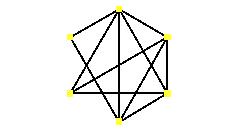}\\
	\includegraphics[height=0.85in]{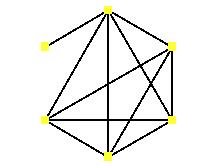}&~~~~~~~
	\includegraphics[height=0.85in]{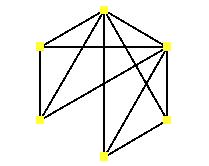}~~~~~~~~~~~
	\includegraphics[height=0.8in]{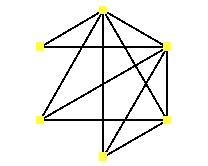}\\
	\end{align*}
	\caption{Non isomorphic graphs with 6 vertices and 11 edges.}
	\label{fig2}
\end{figure}\\\\
Case IV: If $n=7$, then $H_{1}$ is a bipartite graph with 9 vertices and 16 edges. Also, we must have $\gamma_{1}(H_{1})\ge\dfrac{60}{7}$. Using Maple, we see that there is no bipartite graph with 9 vertices and 16 edge such that $\gamma_{1}(H_{1})\ge\dfrac{60}{7}$.\\\\Hence $H$ at least $r+2$ components. Let $H_{1}$ be a component of $H$ with maximum number of vertices. Then $|V(H_{1})|\le n+1$.\\\\
Case V: If $|V(H_{1})|=n+1$, then $H\cong H_{1}\cup(r+1)K_{1}$, where $H_{1}$ is a connected graph on $n+1$ vertices and has 2 as one of its signless Laplacian eigenvalue. Hence (a) holds.\\\\    
Case VI: If $|V(H_{1})|=n$, then either $H\cong H_{1}\cup(r+2)K_{1}$ or $H\cong H_{1}\cup K_{2}\cup rK_{1}$. Suppose $H\cong H_{1}\cup(r+2)K_{1}$ then $\gamma_{1}(G)=\gamma_{1}(H_{1})\ge\dfrac{4(m+1)}{n}\ge \dfrac{2(n-2)(n-3)+24}{n}>n$, and so 	
by Lemmas \ref{l2} and \ref{lq}, $G$ is a non bipartite graph. This is not possible, because $H$ has $r+2$ bipartite components and by Lemma \ref{dlc}, $G$ must be bipartite. Hence (b) holds.\\\\
Case VII: If $|V(H_{1})|= n-1$, then \\ $$H
\cong \left\{ \begin{array}{c}
H_{1}\cup(r+3)K_{1};\\
  H_{1}\cup K_{2}\cup(r+1)K_{1}~~ \text{or}~~ H_{1}\cup K_{2}\cup K_{2}\cup (r-1)K_{1} ;\\
 H_{1}\cup K_{1,2} \cup rK_{1}~~ \text{or}~~ H_{1}\cup K_{3} \cup r K_{1}.  \end{array}\right.$$
 By Lemma \ref{dlc}, it is straight forward that the case $H_{1}\cup(r+3)K_{1}$ is not possible.
 Suppose $H
 \cong 
 H_{1}\cup K_{2}\cup(r+1)K_{1}$ then $\gamma_{1}(G)=\gamma_{1}(H_{1})>n$, and so 	
 by Lemmas \ref{l2} and \ref{lq}, $G$ is a non bipartite graph. This is impossible, because $H$ has $r+2$ bipartite components and by Lemma \ref{dlc}, $G$ must be bipartite. Now, if $H
 \cong 
 H_{1}\cup K_{3} \cup r K_{1}$, then $\gamma_{1}(H_{1})\ge\dfrac{4(m-2)}{n-1}>n-1$. Hence by Lemma \ref{dlc}, $H_{1}$ is a non bipartite graph. Therefore, $H$ has $r$ bipartite components, a contradiction to Lemma \ref{dlc}.\\\\
 Case VIII:  If $|V(H_{1})|= n-2$, then $H$ has at most $\dfrac{(n-2)(n-3)}{2}+6$ edges.  Since $H$ must have at least $\dfrac{(n-2)(n-3)}{2}+6$ edges, it follows that $H \cong K_{n-2}\cup K_{4}\cup rK_{1}$. In this case, for $n\neq 4$, it can be seen that both $H$ and $G\cup K_{2}\cup rK_{1}$ have different number of bipartite components and hence by Lemma \ref{dlc}, the case $|V(H_{1})|=n-2$ is not possible. Hence the theorem is proved.     
\end{proof}

\begin{theorem}
For $n\ge 4$, the graph $K_{n}\cup K_{2}\cup rK_{1}$ is $DQS$.	
\end{theorem}
\begin{proof}
Let $H$ be a graph $Q$-cospectral with $K_{n}\cup K_{2}\cup rK_{1}$. Since $|E(K_{n})|=n(n-1)/2\ge\dfrac{(n-2)(n-3)}{2}+5$, from the above theorem, we have the following choices for $H$:\\\\
Case I: $H\cong H_{1}\cup K_{2}\cup K_{2}\cup (r-1) K_{1}$. This implies that $2\in Spec_{_Q}(K_{n})$ and so $n=4$. Therefore, $Spec_{Q}(H_{1})=\{6,2,2\}$. This is impossible as there is no graph $H_{1}$ of order 3 with $Spec_{Q}(H_{1})=\{6,2,2\}$.\\\\
Case II: $H\cong H_{1}\cup K_{1,2}\cup K_{2}\cup (r-1) K_{1}$. This is impossible because both 1 and 3 cannot be the $Q$-eigenvalues of a complete graph $K_{n}$.\\\\
Case III: $H\cong H_{1}\cup (r+1) K_{1}$. If $n=4$, then $Spec(H_{1})= \{6,2,2,2,2\}\}$. This is impossible because there is no graph with signless Laplacian spectrum as specified above. Suppose $n>4$. Then $H_{1}$ is a non bipartite graph with three distinct $Q$-eigenvalues $2(n-1)$, $n-2$ and $2$. From Theorem \ref{mul}, we see that $H_{1}$ must be one among the graphs $K_{n+1}-e$, $\overline{K_1\cup 2K_3}$ and $\overline{K_{3}\cup S_{4}}$. Since $Spec_{_Q}(K_{n+1}-e)=\{(3n-3+
	\sqrt{(n+1)^2 + 4n - 8})/2, n-2, n-2,\ldots, n-2,(3n-3-
	\sqrt{(n+1)^2 + 4n - 8})/2\}$ and $Spec_{Q}(\overline{K_1\cup 2K_3})=Spec_{Q}(\overline{K_3\cup K_4})=\{9,4,\ldots,4,1\}$. This case is not possible.\\\\   
Case IV: $H\cong H_{1}\cup K_{2}\cup  r K_{1}$. This implies that $H$ and $K_{n}$ are $Q$-cospectral mates. Since $K_{n}$ is $DQS$, we must have $H\cong K_{n}$. This completes the proof.	
\end{proof}
\begin{rmk}
The above theorem is not true for $n=3$, for example, the graphs $K_{3}\cup K_{2}\cup rK_{1}$ and $K_{1,3}\cup K_{2}\cup (r-1)K_{1}$ are $Q$-cospectral but not isomorphic.
\end{rmk}
In \cite{cou}, it is claimed that
if $G$ is a $DQS$ connected non-bipartite graph with $n\ge 3$ vertices and if $H$ is $Q$-cospectral with
$G \cup r K_1 \cup sK_2$, then $H$ is a $DQS$ graph. Here we give a counterexample for the claim. Let  $G$ and $H_{1}$ be graphs as depicted in Fig. \ref{fig3}. Then it is easy to see that $H_{1}$ is $DQS$ connected non-bipartite graph and the $Q$-spectrum of $G$ and $H_{1}$ are respectively,  
$\{(7+\sqrt{17})/2,3,(7-\sqrt{17})/2,1,1\}$ and $\{(7+\sqrt{17})/2,4,3,(7-\sqrt{17})/2,1,1\}$.  Thus the graphs $H_{1}\cup K_{2}\cup K_{2}$ and $G\cup C_{4}\cup K_{1}$  are $Q$-cospectral. Hence the claim is false.\\   \begin{figure}[h]
	\begin{align*}
	\includegraphics[height=1.1in]{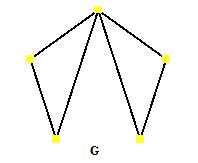}&~~~
	\includegraphics[height=1.1in]{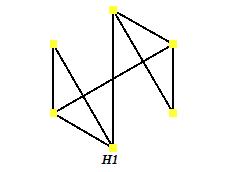}
	\end{align*}
	\caption{Graphs $G$ and $H_{1}$.}
	\label{fig3}
\end{figure}
\section{Signless Laplacian spectral determination of the join $K_{n-\alpha}-e\vee \alpha K_{1}$}
In Section \ref{section}, we showed that the graph $K_{n}\cup K_{2}\cup rK_{1}$ with $n\ge 4$ is $DQS$. In this section, we study the signless Laplacian spectral characterization of its complements. we first list some necessary lemmas required to prove our main results. We denote the eigenvalues of a Hermitian matrix $M$ of order $m$ by $\theta_{1}(M)\ge \theta_{2}(M)\ge\ldots\ge\theta_{m}(M)$. The join of two graphs $G$ and $H$, denoted by, $G\vee H$ is the graph obtained by joining every vertex of $G$ to every vertex of $H$. 

\begin{lemma}[\cite{book}]\label{jlh}
	Let $M=N+P$, where N and P are Hermitian matrices of order n. Then for $1\le i, j \le n$, we have
	\begin{enumerate}[(i)]
		\item 
		$\theta_{i}(N)+\theta_{j}(P)\le \theta_{i+j-n}(M)~(i+j>n)$;
		\item 
		$\theta_{i+j-1}(M)\le \theta_{i}(N)+\theta_{j}(P)~(i+j-1\le n)$.
	\end{enumerate}
\end{lemma}
We denote by $P_{Q}(G, \lambda)$, the characteristic polynomial of $Q(G)$.
\begin{lemma}[\cite{js}]
Let $G=K_{n_{1},n_{2},\ldots,n_{r}}$ be a complete multipartite graph with $\ds\sum_{i=1}^{r}n_{i}=n$. Then\\
$P_{Q}(G,\lambda)=\left(1-\ds\sum_{i=1}^{r}\dfrac{n_{i}}{\lambda-n+2n_{i}}\right)\ds\prod_{i=1}^{r}(\lambda-n+n_{i})^{n_{i}-1}(\lambda-n+2n_{i})$.
\end{lemma}
\begin{rmk}\label{jls}
From the above lemma, it can be seen that the $Q$-spectrum of $K_{n-\alpha}-e\vee \alpha K_{1}$ consists of:
\begin{enumerate}[a.]
	\item 
	$n-\alpha$, with multiplicity $\alpha-1$;
	\item 
	$n-2$, with multiplicity $n-\alpha-2$;
	\item
	and the three roots of the polynomial
	\begin{equation}\label{jequa1}
x^3+(2\alpha-4n+6)x^2+(2\alpha^2-6\alpha n+5n^2+10\alpha-16n+12)x-2(n-\alpha-2)(n^2-(\alpha+3)n+4\alpha).
	\end{equation}
\end{enumerate} 
\end{rmk}
The following two lemmas can be obtained easily.
\begin{lemma}\label{jequam}
The signless Laplacian spectrum of $\overline{K_{1,3}\cup K_{2}\cup (n-\alpha-3)K_{1}}$ consists of 
\begin{enumerate}[a.]
	\item 
	$n-3$, with multiplicity $2$;
	\item 
	$n-2$, with multiplicity $n-5$;
	\item
	and the three roots of the polynomial
	\begin{equation*}
	x^3+(-4n+12)x^2+(5n^2-34n+60)x-2n^3+22n^2-84n+120.
	\end{equation*}
\end{enumerate} 
\end{lemma}
\begin{lemma}\label{jm}
The signless Laplacian spectrum of $\overline{K_{2,4}-e \cup (n-6)K_{1}}$ consists of 
\begin{enumerate}[a.]
	\item 
	$n-4$, with multiplicity $2$;
	\item 
	$n-2$, with multiplicity $n-6$;
	\item
	and the four roots of the polynomial
	\begin{eqnarray*}
&x^4+(-5n+18)x^3+(9n^2-70n+140)x^2+(-7n^3+86n^2-360n+522)x\\
&+2n^4-34n^3+220n^2-644n+708.~~~~~~~~~~~~~~~~~~~~~~~~~~~~~~~~~~~~~~~~~~~~~~
	\end{eqnarray*}
\end{enumerate} 
\end{lemma}
\begin{dfn}[\cite{j7}]
	A connected bipartite graph G with vertex partition sets U and V is said to balanced if the cardinalities of U and V are same. 
\end{dfn}
\begin{lemma}[\cite{j7}]\label{ljv}
	Let  G be a graph of order $n\ge2$. Then $\gamma_{2}(G)\le n-2$. Moreover, $\gamma_{k+1}(G)=n-2~ (1\le k\le n-1)$ if and only if $\overline{G}$ has either k balanced bipartite components or k+1 bipartite components.
\end{lemma}
\begin{lemma}[\cite{kite}]\label{jldn}
	Let G be a graph on n vertices. Then $d_{n-1}(G)\ge \gamma_{n-1}(G)-1$. Furthermore, if the equality holds, then $d_{n-1}(G)=d_{n}(G)$.
\end{lemma}
\begin{lemma}[\cite{split}]\label{jld3}
	Let G be a graph on n vertices. Then $\gamma_{3}(G)\ge d_{3}(G)-\sqrt{2}$.
\end{lemma}
\begin{theorem}
Let $n-\alpha> 3$. Then for $\alpha\neq 3$, the graph $K_{n-\alpha}-e\vee \alpha K_{1}$ is $DQS$. If $\alpha=3$, then the only Q-cospectral mate of $K_{n-3}-e\vee \alpha K_{1}$ is $\overline{K_{1,3}\cup K_{2}\cup (n-6)K_{1}}$.	
\end{theorem}
\begin{proof}
Let $G\cong K_{n-\alpha}-e\vee \alpha K_{1}$. Let $H$ be a graph $Q$-cospectral with $K_{n-\alpha}-e\vee \alpha K_{1}$. Then it is easy to see that 
\begin{eqnarray}
\displaystyle\sum_{i=1}^{n}\gamma_{i}(\overline{H})=2|E(\overline{H})|=\sum_{i=1}^{n}d_{i}(\overline{H})=\sum_{i=1}^{n}d_{i}(\overline{G})=\alpha(\alpha-1)+2,~~~~~~~~~~~~~~~~ \label{jequa3}\\
{\displaystyle\sum_{i=1}^{n}\gamma^{2}_{i}(\overline{H})=\sum_{i=1}^{n}d_{i}(\overline{H})(d_{i}(\overline{H})+1)=\sum_{i=1}^{n}d_{i}(\overline{G})(d_{i}(\overline{G})+1)=\alpha^3-\alpha^2+4.~~~
	\label{jequa4}}\end{eqnarray}
Let $\overline{H}$ be the complement graph of $H.$ If $\alpha=1$, then by (\ref{jequa3}), $|E(\overline{H})|=1$ and so $\overline{H}\cong K_{2}\cup (n-\alpha-2) K_{1}$. Hence the theorem is true for $\alpha=1$. If $\alpha=2$. Then from (\ref{jequa3}), $|E(\overline{H})|=2$ and so
$\overline{H}\cong P_{3}\cup (n-3)K_{1}~ \text{or}~ K_{2}\cup K_{2}\cup (n-4)K_{1}$. By (\ref{jequa4}), $\overline{H}$ must be isomorphic to $K_{2}\cup K_{2}\cup (n-4)K_{1}$. Thus the result is true for $\alpha=2$.\\\\ Let $\alpha\ge 3$. we divide the proof into two cases, when the polynomial (\ref{jequa1}) has a root in $(n-2, n-\alpha]$ or otherwise. Suppose the polynomial (\ref{jequa1}) has a root in $(n-2, n-\alpha]$. Then by Remark \ref{jls} the $Q$-spectrum of $H$ consists of $\gamma_{1}(H),\gamma_{2}(H)=\gamma_{3}(H)=\ldots=\gamma_{n-\alpha-1}=n-2, \gamma_{n-\alpha}(H),\gamma_{n-\alpha+1}(H)=\gamma_{n-\alpha+2}(H)=\ldots=\gamma_{n-1}(H)=n-\alpha,\gamma_{n}(H)$, where $\gamma_{1}(H)$, $\gamma_{n-\alpha}(H)$ and $\gamma_{n}(H)$ are the roots of the equation (\ref{jequa1}).  From Lemma \ref{jlh}, we have
\begin{align*}
&n-2\le \gamma_{n-\alpha-1}(H)+\gamma_{\alpha+3}(\overline{H})\le\theta_{2}(Q(H)+Q(\overline{H}))=n-2,\\\\[2mm]
&n-2=\theta_{n}(Q(H)+Q(\overline{H}))\le \gamma_{n-\alpha+1}(H)+\gamma_{\alpha}(\overline{H})=n-\alpha+\gamma_{\alpha}(\overline{H})\\\\[2mm]
&\text{and}\\\\
&n-\alpha+\gamma_{3}(\overline{H})=\gamma_{n-1}(H)+\gamma_{3}(\overline{H})\le\theta_{2}(Q(H)+Q(\overline{H}))=n-2.
\end{align*} 
Thus
\begin{equation}
\renewcommand{\arraystretch}{1.2}
\left.\begin{array}{r@{\;}l}
&\gamma_{3}(\overline{H})=\gamma_{4}(\overline{H})=\cdots=\gamma_{\alpha}(\overline{H})=\alpha-2,\\\\[2mm]
&\gamma_{\alpha+3}(\overline{H})=\gamma_{\alpha+4}(\overline{H})=\cdots=\gamma_{n}(\overline{H})=0.
\end{array}\right\} \label{jequa2}
\end{equation}
We now claim that $\overline{H}$ has a component $H_{\ast}$ with $\gamma_{2}(H_{\ast})\ge\alpha-2$. Suppose $\overline{H}$ has exactly $k$ connected components $H_{1}, H_{2},\ldots, H_{k}$. If the claim is not true, i.e., $\gamma_{2}(H_{i})< \alpha-2$, for all $1\le i\le k$. Then there exists $\alpha$ components say $H_{i_{_1}},H_{i_{_2}},\ldots,H_{i_{_\alpha}}$ such that
 $\gamma_{1}(H_{i_{_1}})=\gamma_{1}(\overline{H})\ge \alpha-2,\gamma_{1}(H_{i_{_2}})=\gamma_{2}(\overline{H})\ge \alpha-2,\gamma_{1}(H_{i_{_3}})=\ldots= \gamma_{1}(H_{i_{_\alpha}})= \alpha-2$, by (\ref{jequa2}) and also we must have $\alpha\ge4$.
Since $\gamma_{2}(H_{i_{j}})<\alpha-2$, $1\le j\le \alpha$ and $\overline{H}$ can have at most $2$    
non-zero signless Laplacian eigenvalues strictly less than $\alpha-2$, there exists at least one connected component say $H_{i_{_j}}$ such that 
$\gamma_2(H_{i_j}) = 0$. Therefore $H_{i_j}\cong K_{2}$, since $\gamma_2(H_{i_j}) = 0$ and 0 can be a
signless Laplacian eigenvalue of a connected graph with multiplicity at most 1. Hence $\alpha-2\le \gamma_{1}(H_{i_{j}})=2$ and so $\alpha=4$. Thus $|E(\overline{H})|=7$ by (\ref{jequa3}) and $\overline{H}=H_{i_{_1}}\cup H_{i_{_2}}\cup K_{2}\cup K_{2}\cup H_{0}$, for some graph $H_{0}$ (not necessarily connected). Since the maximum number of non-zero signless Laplacian eigenvalue of $\overline{H}$ is  $6$, one can easily check that there is no graph $H_{i_{_1}}\cup H_{i_{_2}}\cup H_{0}$ with 5 edges, $\gamma_2(H_{i_{1}})<2$ and $\gamma_2(H_{i_{2}})<2$. Thus our claim is true.\\\\
 Let $H_{1}$ be a connected component of $\overline{H}$ with 
$\gamma_{2}(H_{1})\ge \alpha-2$, then from Lemma \ref{ljv},
 $|V(H_{1})|\ge \gamma_{2}(H_{1})+2\ge\alpha$. Now since $\gamma_{n-\alpha-1}(H)=n-2$ and $\gamma_{n-\alpha}(H)\neq n-2$, by Lemma \ref{ljv}, it follows that $\overline{H}$ has exactly $n-\alpha-2$ or $n-\alpha-1$ bipartite components. Let us first assume that $\overline{H}$ has exactly $n-\alpha-2$ bipartite components. Since $\gamma_{n-\alpha-1}(H)=n-2$, by Lemma \ref{ljv}, all the bipartite graphs of $\overline{H}$ are balanced. Now as an isolated vertex is not balanced and also since $|V(H_{1})|\ge \alpha$, we have
 $$\overline{H}\cong\left\{\begin{array}{r@{\;}l}
 &H_{1}\cup K_{3}, \text{$H_{1}$ is a balanced bipartite graph of order $\alpha$},\\[2mm]
 &H_{1}\cup P_{4}, \text{$H_{1}$ is a balanced bipartite graph of order $\alpha$},\\[2mm]
 &H_{1}\cup C_{4}, \text{$H_{1}$ is a balanced bipartite graph of order $\alpha$},\\[2mm]
 &H_{1}\cup K_{2}\cup K_{2}\cup K_{2}, \text{$H_{1}$ is a balanced bipartite graph of order $\alpha$},\\[2mm]
 &H_{1}\cup K_{2}\cup K_{2}, \text{$H_{1}$ is a balanced bipartite graph of order $\alpha+1$},\\[2mm]
  &H_{1}\cup K_{2}, \text{$H_{1}$ is a balanced bipartite graph of order $\alpha+2$},\\[2mm]
   &H_{1}, \text{$H_{1}$ is a balanced bipartite graph of order $\alpha+3$},\\[2mm]
    &H_{1}\cup K_{2}\cup K_{2}, \text{$H_{1}$ is a non bipartite graph of order $\alpha$},\\[2mm]
     &H_{1}\cup K_{2}, \text{$H_{1}$ is a non bipartite graph of order $\alpha+1$},\\[2mm]
      &H_{1}, \text{$H_{1}$ is a non bipartite graph of order $\alpha+2$}.\\[2mm]
 \end{array}\right.$$
Case I: Let $\overline{H}\cong H_{1}\cup K_{3}$ or $H_{1}\cup P_{4}~ \text{or}~ H_{1}~\cup ~C_{4}$, $H_{1}$ is a balanced bipartite graph of order $\alpha$. Then $\alpha$ is even and $\gamma_{\alpha+1}(\overline{H})+\gamma_{\alpha+2}(\overline{H})\le 4$. If $\alpha=4$,  then $H_{1}\cong C_{4} ~\text{or}~ P_{4}$. In this case, $\overline{H}$ does not satisfy (\ref{jequa4}), a contradiction. Hence $\alpha\ge 6$. Now  using Lemmas \ref{ljv}, \ref{lq}, \ref{l2} and by (\ref{jequa2}) and (\ref{jequa3}), we get 
$$3\alpha-2=\gamma_{\alpha+1}(\overline{H})+\gamma_{\alpha+2}(\overline{H})+\gamma_{2}(\overline{H})+\gamma_{1}(\overline{H})\le 2\alpha+2.$$ This implies that $\alpha\le 4$, a contradiction.\\\\
Case II: Let  $\overline{H}\cong H_{1}\cup K_{2}\cup K_{2}\cup K_{2}$, \text{$H_{1}$ is a balanced bipartite graph of order $\alpha$}. Then $\alpha-2=\gamma_{\alpha}(\overline{H})= 2$ and so $\alpha= 4$. Since $H_{1}$ is a balanced bipartite graph of order $\alpha$, we must have  $H_{1}\cong C_{4}$ or $P_{4}$. In either case, we get a contradiction to (\ref{jequa4}).   \\\\      
Case III: Let $\overline{H}\cong H_{1}\cup K_{2}\cup K_{2}$, $H_{1}$ is a balanced bipartite graph of order $\alpha+1$. Then $\gamma_{\alpha+1}(\overline{H})+\gamma_{\alpha+2}(\overline{H})\le 4$. Using Lemmas \ref{ljv}, \ref{lq}, \ref{l2} and by (\ref{jequa2}) and (\ref{jequa3}), we get 
$$3\alpha-2=\gamma_{\alpha+1}(\overline{H})+\gamma_{\alpha+2}(\overline{H})+\gamma_{2}(\overline{H})+\gamma_{1}(\overline{H})\le 2\alpha+4.$$ 
Thus $3\le\alpha\le6$. Since $H_{1}$ is a balanced bipartite graph of order $\alpha+1$, we must have $\alpha=3$ or 5. 
If $\alpha=3$, then by (\ref{jequa3}), $|E(\overline{H})|=4$. So, $H_{1}$ is a bipartite graph with 4 vertices and 2 edges. This is impossible. If  $\alpha=5$, then by (\ref{jequa3}), $|E(\overline{H})|=11$. So, $H_{1}\cong K_{3,3}$. This contradicts (\ref{jequa4}).\\\\
Case IV: Let $\overline{H}\cong H_{1}\cup K_{2}$, $H_{1}$ is a balanced bipartite graph of order $\alpha+2$. Then by (\ref{jequa3}), $\alpha(\alpha-1)=2|E(\overline{H})|\le (\alpha+2)^{2}/2$. Thus $3\le\alpha\le 6$. Since $H_{1}$ is a balanced bipartite graph of order $\alpha+2$, we must have $\alpha=4$ or $6$. If $\alpha=4$, then  by (\ref{jequa3}), $|E(\overline{H})|=7$. Hence $H_{1}$ is a balanced bipartite graph with 6 vertices and 6 edges. Therefore $H_{1}$ must be a graph with degree sequence $(2,2,\ldots,2)$ or $(3,2,2,2,2,1)$ or $(3,3,2,2,1,1)$. But none of them satisfy (\ref{jequa4}). If $\alpha=6$, then  by (\ref{jequa3}), $|E(\overline{H})|=16$. Hence $H_{1}$ is a balanced bipartite graph with 8 vertices and 15 edges. Thus $H_{1}\cong K_{4,4}-e$. This is a contradiction to (\ref{jequa4}).\\\\
Case V: Let $\overline{H}\cong H_{1}$, $H_{1}$ is a balanced bipartite graph of order $\alpha+3$ . Then by (\ref{jequa3}), $\alpha(\alpha-1)+2=2|E(H_{1})|\le (\alpha+3)^{2}/2$. Thus $3\le\alpha\le 8$. Since $H_{1}$ is a balanced bipartite graph of order $\alpha+3$, we must have $\alpha=3$, $5$ or $7$. If $\alpha=3$, then $E(\overline{H})=4$. Hence $H_{1}$ is a  connected balanced bipartite graph with 6 vertices and 4 edges. This is impossible. If $\alpha=5$, then $E(\overline{H})=11$. Thus $H_{1}$ is a balanced bipartite graph with 8 vertices and 11 edges. So the degree sequence of $H_{1}$ is one among (4,4,4,3,2,2,2,1), (4,4,3,3,3,2,2,1),(4,3,3,3,3,3,2,1), (3,3,3,3,3,3,2,2), (4,3,3,3,3,2,2,2), (4,4,3,3,2,2,2,2) and (4,4,3,3,3,3,1,1). But none of them satisfy (\ref{jequa4}). If $\alpha=7$, then $|E(\overline{H})|=22$ and so $H_{1}$ is a balanced bipartite graph with 10 vertices and 22 edges. Thus the degree sequence of  $H_{1}$ is one among (3,3,4,4,5,5,5,5,5,5), (3,4,4,4,4,5,5,5,5,5), (2,4,4,4,5,5,5,5,5,5) and (4,4,4,4,4,4,5,5,5,5). This contradicts (\ref{jequa4}).\\\\
Case VI: Let $\overline{H}\cong H_{1}\cup K_{2}\cup K_{2}$, $H_{1}$ is a non bipartite graph of order $\alpha$. Suppose $\gamma_{\alpha+1}(\overline{H})=\gamma_{\alpha+2}(\overline{H})=2=\gamma_{1}(K_{2})$. Then 
by (\ref{jequa2}) and from Lemma \ref{ljv}, $\alpha-2=\gamma_{\alpha}(H_{1})\le\gamma_{2}(H_{1})\le\alpha-2$. Thus $\gamma_{\alpha}(H_{1})=\gamma_{\alpha-1}(H_{1})=\ldots\gamma_{2}(H_{1})=\alpha-2$ and so by Lemma \ref{ljv}, $H_{1}\cong K_{\alpha}$. This contradicts (\ref{jequa4}). If either $\gamma_{\alpha+1}(\overline{H})\neq 2$ or $\gamma_{\alpha+2}(\overline{H})\neq 2$, then by (\ref{jequa2}), $\alpha=4$ and so $H_{1}\cong K_{4}-e$. This is not possible by (\ref{jequa4}).\\\\   
Case VII: Let $\overline{H}\cong H_{1}\cup K_{2}$, $H_{1}$ is a non bipartite graph of order $\alpha+1$. Then $n-\alpha=3$, a contradiction to the hypothesis.\\\\ 
Case VIII: Let $\overline{H}\cong H_{1}$, $H_{1}$ is a non bipartite graph of order $\alpha+2$.
Then $n-\alpha=2$, a contradiction.\\\\
Now, suppose $\overline{H}$ has exactly $n-\alpha-1$ bipartite graphs. Then $\gamma_{\alpha+2}(\overline{H})=0$ and 
$$\overline{H}\cong\left\{\begin{array}{r@{\;}l}
&H_{1}\cup K_{2}\cup K_{2}\cup (n-\alpha-4)K_{1}, \text{$H_{1}$ is a bipartite graph of order $\alpha$},\\[2mm]
&H_{1}\cup P_{3}\cup (n-\alpha-3)K_{1}, \text{$H_{1}$ is a bipartite graph of order $\alpha$},\\[2mm]
&H_{1}\cup K_{2}\cup (n-\alpha-3)K_{1}, \text{$H_{1}$ is a bipartite graph of order $\alpha+1$},\\[2mm]
&H_{1}\cup (n-\alpha-2)K_{1}, \text{$H_{1}$ is a bipartite graph of order $\alpha+2$},\\[2mm]
&H_{1}\cup K_{2}\cup (n-\alpha-2)K_{1}, \text{$H_{1}$ is a non bipartite graph of order $\alpha$},\\[2mm]
&H_{1}\cup  (n-\alpha-1)K_{1}, \text{$H_{1}$ is a non bipartite graph of order $\alpha+1$}.\\[2mm]
\end{array}\right.$$
Case I: Let $\overline{H}\cong H_{1}\cup K_{2}\cup K_{2}\cup (n-\alpha-4)K_{1}$, \text{$H_{1}$ is a bipartite graph of order $\alpha$}. Then from (\ref{jequa2}), $\alpha=4$  and so by (\ref{jequa3}), $H_{1}$ is a bipartite graph with 4 vertices and 5 edges. This is impossible.\\\\
Case II: Let $\overline{H}\cong H_{1}\cup P_{3}\cup (n-\alpha-3)K_{1}$, \text{$H_{1}$ is a bipartite graph of order $\alpha$}. Then $\alpha-2\le 3$. If $\alpha=4$ or $5$, then from (\ref{jequa3}), $H_{1}$ is either a bipartite graph with 5 vertices and 9 edges or a bipartite graph with 4 vertices and 5 edges. This is impossible. If $\alpha=3$, then by (\ref{jequa3}), $H_{1}\cong P_{3}$, which is a contradiction to (\ref{jequa4}). \\\\
Case III: Let $\overline{H}\cong H_{1}\cup K_{2}\cup(n-\alpha-3)K_{1}$, \text{$H_{1}$ is a bipartite graph of order $\alpha+1$}. Using Lemmas \ref{ljv}, \ref{lq}, \ref{l2} and by (\ref{jequa2}) and (\ref{jequa3}), we get 
$$3\alpha-2=\gamma_{\alpha+1}(\overline{H})+\gamma_{\alpha+2}(\overline{H})+\gamma_{2}(\overline{H})+\gamma_{1}(\overline{H})\le 2\alpha+2.$$ This implies that $3\le\alpha\le4$. If $\alpha=3$, then $H_{1}$ is a bipartite graph with 4 vertices and 3 edges. Thus $H_{1}\cong  K_{1,3}$, by (\ref{jequa4}). From Remark \ref{jls} and by Lemma \ref{jequam}, the graphs $\overline{K_{1,3}\cup K_{2}\cup (n-6)K_{1}}$  and $K_{n-3}-e\vee 3 K_{1}$ are $Q$-cospectral.  If $\alpha=4$, then $H_{1}$ is a bipartite graph with 5 vertices and 6 edges, i.e., $H_{1}\cong K_{2,3}$, which is a contradiction to (\ref{jequa4}).\\\\
Case IV: Let $\overline{H}\cong H_{1}\cup (n-\alpha-2)K_{1}$, \text{$H_{1}$ is a bipartite graph of order $\alpha+2$}. Then by (\ref{jequa3}), $\alpha(\alpha-1)+2=2|E(H_{1})|\le (\alpha+2)^{2}/2$. Therefore $3\le \alpha\le 6$. If $\alpha=3$, then by (\ref{jequa3}), $H_{1}$ is a connected bipartite graph with 5 vertices and 4 edges. Therefore $H_{1}\cong P_{5}$, by (\ref{jequa4}). This is impossible because $\alpha-2=1$ (belonging to $Spec_{Q}(\overline{H})$) is not a signless Laplacian eigenvalue of $P_{5}$. If $\alpha=4$, then by (\ref{jequa4}), $H_{1}$ is a bipartite graph with 6 vertices and 7 edges, and so $H_{1}\cong K_{2,4}-e$ by (\ref{jequa4}). Now by Lemma \ref{jm} and Remark \ref{jequa4} it can be seen that the graphs $\overline{K_{2,4}-e\cup (n-6)K_{1}}$  and $K_{n-4}-e\vee 4 K_{1}$ are not $Q$-cospectral.  If $\alpha=5$, then by (\ref{jequa3}), $H_{1}$ is a bipartite graph with 7 vertices and 11 edges, i.e., $H_{1}\cong K_{3,4}-e$. This cannot be true by (\ref{jequa4}). If $\alpha=6$, then by (\ref{jequa3}), $H_{1}$ is a bipartite graph with 8 vertices and 16 edges, i.e., $H_{1}\cong K_{4,4}$. This is not possible by (\ref{jequa4}).\\\\
Case V: Let $\overline{H}\cong H_{1}\cup K_{2}\cup (n-\alpha-2)K_{1}$, \text{$H_{1}$ is a non bipartite graph of order $\alpha$}. If $\gamma_{\alpha+1}(\overline{H})\neq2=\gamma_{1}(K_{2})$, then $\alpha-2\le 2$ and so $\alpha= 3$ or $4$. If $\alpha=3$, then $H_{1}\cong K_{3}$  and also if $\alpha=4$, then by (\ref{jequa3}), $H_{1}\cong K_{4}$. So we are done. Suppose $\gamma_{\alpha+1}(\overline{H})=2$ then by (\ref{jequa2}), $H_{1}$ is a non bipartite graph of order $\alpha$ and $\alpha-2$ is an $Q$-eigenvalue of $H_{1}$ with multiplicity at least  $\alpha-2$. Therefore $H_{1}\cong K_{\alpha}$ or  by Theorem \ref{mul}, $H_{1}\cong K_{\alpha}-e$. The latter case  contradicts (\ref{jequa4}).\\\\
Case VI: Let $\overline{H}\cong H_{1}\cup (n-\alpha-1)K_{1}$, \text{$H_{1}$ is a non bipartite graph of order $\alpha+1$}. By (\ref{jequa2}) and from Lemmas \ref{jldn} and \ref{jld3}, we have $d_{3}(H_{1})\le \alpha-1$ and $d_{\alpha}(H_{1})\ge \alpha-3$.\\\\
Subcase I: Suppose $d_{\alpha}(H_{1})=d_{\alpha+1}(H_{1})$ and $x_{i}$ is the number of vertices of $H_{1}$ of degree $\alpha-i$, $i=0,1,2,3$. Then by (\ref{jequa3}) and (\ref{jequa4}), 
\begin{equation}\label{jequa5}
\renewcommand{\arraystretch}{1.2}
\left.\begin{array}{r@{\;}l}
&x_{0}+x_{1}+x_{2}+x_{3}=\alpha+1,\\[2mm]
&\alpha x_{0}+(\alpha-1)x_{1}+(\alpha-2)x_{2}+(\alpha-3)x_{3}=\alpha(\alpha-1)+2,\\[2mm]
&\alpha^2 x_{0}+(\alpha-1)^2x_{1}+(\alpha-2)^2x_{2}+(\alpha-3)^2x_{3}=\alpha(\alpha-1)^2+2.\\[2mm]
\end{array}\right\} 
\end{equation}  
Solving (\ref{jequa5}), for $x_{1}$, $x_{2}$ and $x_{3}$, we obtain
\begin{equation}\label{jequa6}
\renewcommand{\arraystretch}{1.2}
\left.\begin{array}{r@{\;}l}
&x_{1}=(1/2)\alpha^2-(7/2)\alpha-3x_{0}+9,\\[2mm]
&x_{2}=-\alpha^2+3x_{0}+8\alpha-13,\\[2mm]
&x_{3}=(1/2)\alpha^2-x_{0}-(7/2)\alpha+5.\\[2mm]
\end{array}\right\} 
\end{equation}  
Since $d_{3}(H_{1})\le \alpha-1$, we have $x_{0}=0,1,2$ and so from (\ref{jequa6}), it is easy to see that $x_{1}$, $x_{2}$ and $x_{3}$ are all non negative if and only if 
\begin{enumerate}[a.]
\item	
$x_{0}=0~ \text{and}~ \alpha=5~ (x_{1}=4, x_{2}=2, x_{3}=0)$;
\item 
$x_{0}=1~ \text{and}~ \alpha=6~ (x_{1}=3, x_{2}=2, x_{3}=1)$;
\item 
$x_{0}=2~ \text{and}~ \alpha=6~ (x_{1}=0, x_{2}=5, x_{3}=0)$;
\item 
$x_{0}=2~ \text{and}~ \alpha=7~ (x_{1}=3, x_{2}=0, x_{3}=3)$.
\end{enumerate}
Thus the degree sequence of $H_{1}$ is one among the following: (4,4,4,4,3,3), (6, 5,5,5,4,4,3), (6,6,4,4,4,4,4) and (7,7,6,6,6,4,4,4). Hence  $H_{1}$ must be one of the  graphs as depicted in Fig. \ref{jf1}. 
\begin{figure}[t]
	\begin{align*}
	\includegraphics[height=0.85in]{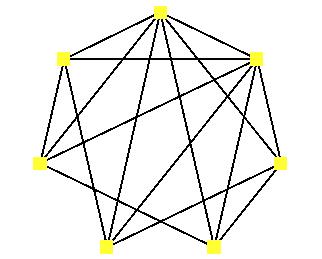}&~~~
	\includegraphics[height=0.85in]{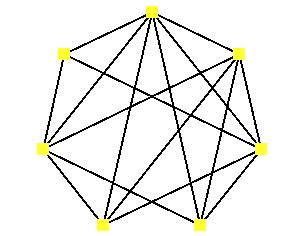}~~~
	\includegraphics[height=0.85in]{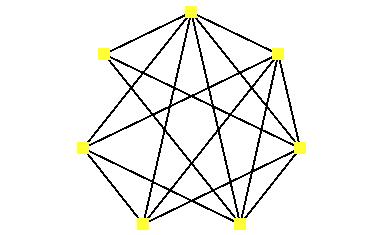}\\&~~~~\includegraphics[height=1.15in]{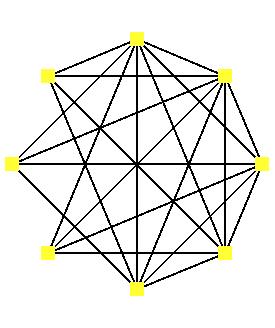}\\
	\includegraphics[height=0.85in]{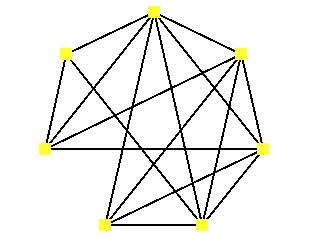}&~
	\includegraphics[height=0.85in]{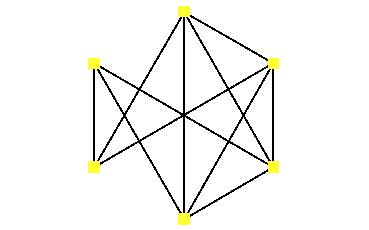}
	\includegraphics[height=0.85in]{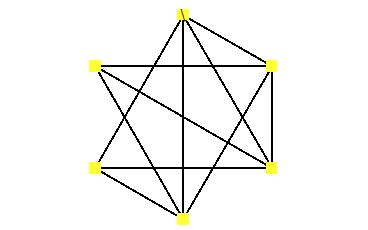}\\
	\end{align*}
	\caption{Non isomorphic graphs with degree sequence (4,4,4,4,3,3), (6, 5,5,5,4,4,3), (6,6,4,4,4,4,4) and (7,7,6,6,6,4,4,4).}
	\label{jf1}
\end{figure}\\\\
Computing the signless Laplacian spectrum of the graphs as depicted in Fig. \ref{jf1}, it can be seen that none of these graphs have $\alpha-2$ as its $Q$-eigenvalue with multiplicity at least $\alpha-2$, which contradicts (\ref{jequa2}).\\\\
Subcase II: Suppose $d_{\alpha}(H_{1})\neq d_{\alpha+1}(H_{1})$. Assume that $d_{\alpha+1}(H_{1})\ge \alpha-2$. Then $d_{\alpha}(H_{1})=d_{\alpha-1}(H_{1})=\ldots=d_{3}(H_{1})=\alpha-1$.  Thus $$2\alpha-2\le d_{1}(H_1)+d_{2}(H_{1})=\alpha(\alpha-1)+2-(\alpha-2)(\alpha-1)-\alpha+2=\alpha+2.$$ Therefore $\alpha=3,4$. If $\alpha=3$, then $H_{1}$ is a graph with degree sequence $(3,2,2,1)$. This contradicts (\ref{jequa4}). If $\alpha=4$, then $H_{1}$ is a graph with degree sequence $(3,3,3,3,2)$. This is not possible by (\ref{jequa4}). Hence $d_{\alpha+1}\le \alpha-3$.\\\\
Let $x_{i}$, $i=0,1,2$ be the number of vertices of $H_{1}$ with degree $\alpha-i$. Then
\begin{equation}
\renewcommand{\arraystretch}{1.2}
\left.\begin{array}{r@{\;}l}
&x_{0}+x_{1}+x_{2}=\alpha,\\[2mm]
&\alpha x_{0}+(\alpha-1)x_{1}+(\alpha-2)x_{2}+d_{\alpha+1}=\alpha(\alpha-1)+2, \\[2mm]
&\alpha^2 x_{0}+(\alpha-1)^2x_{1}+(\alpha-2)^2x_{2}+d_{\alpha+1}^2=\alpha(\alpha-1)^2+2.\\[2mm]
\end{array}\right\} \label{jequa7}
\end{equation} 
Solving \ref{jequa7}, we obtain
\begin{equation}
\renewcommand{\arraystretch}{1.2}
\left.\begin{array}{r@{\;}l}
&x_{0}=-(1/2)d_{\alpha+1}^2+\alpha d_{\alpha+1}-(3/2)d_{\alpha+1}-2\alpha+4,\\[2mm]
&x_{1}= d_{\alpha+1}^2-2\alpha d_{\alpha+1}+2d_{\alpha+1}+5\alpha-6. \\[2mm]
\end{array}\right\}\label{jequa8} 
\end{equation}
Since $d_{\alpha+1}\le \alpha-3$,
$x_{0}\ge \alpha(d_{\alpha+1}/2-2)+4$, by (\ref{jequa8}). Therefore $d_{\alpha+1}=3$, since $x_{0}=0,1,2$. Hence  $x_{0}=\alpha-5$ and $x_{1}=9-\alpha$, by (\ref{jequa8}). Now as $x_{0}=0,1,2$ and $d_{\alpha+1}\le \alpha-3$, we must have either $(\alpha=6, x_{0}=1, x_{1}=3, x_{2}=2)$ or $(\alpha=7, x_{0}=2, x_{1}=2,x_{2}=3)$. Therefore $H_{1}$ is a graph with degree sequence (6,5,5,5,4,4,3) or (7,7,6,6,5,5,5,3). Hence  $H_{1}$ must be one of the  graphs as depicted in Fig. \ref{jf2}. 
\begin{figure}[h]
	\begin{align*}
	\includegraphics[height=0.8in]{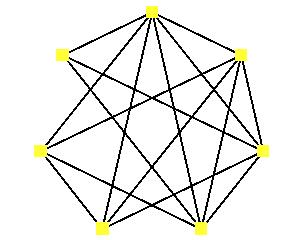}&~
	\includegraphics[height=0.8in]{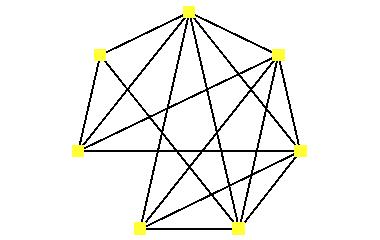}&~
	\includegraphics[height=0.85in]{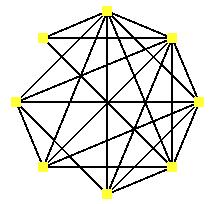}&~~~~~~
	\includegraphics[height=0.85in]{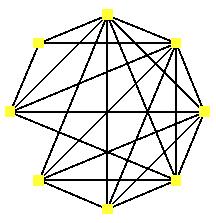}\\
	\end{align*}
	\caption{Non isomorphic graphs with degree sequence (6,5,5,5,4,4,3) or (7,7,6,6,5,5,5,3). }
	\label{jf2}
\end{figure}Since none of these graphs has $\alpha-2$ as its $Q$-eigenvalue with multiplicity at least $\alpha-2$ we get a contradiction to (\ref{jequa2}). Similarly the theorem can be proved when the polynomial (\ref{jequa1}) does not have root in $(n-2, n-\alpha]$. This completes the proof.
\end{proof}	
\label{Bibliography}
\bibliographystyle{amsplain}  
\bibliography{comingblb}  

\providecommand{\bysame}{\leavevmode\hbox to3em{\hrulefill}\thinspace}
\providecommand{\MR}{\relax\ifhmode\unskip\space\fi MR }
\providecommand{\MRhref}[2]{%
  \href{http://www.ams.org/mathscinet-getitem?mr=#1}{#2}
}
\providecommand{\href}[2]{#2}
\begin{thebibliography}{10}

\bibitem{cou}
A.~Z. Abdian, A.~Behmaram, and G.~H. Fath-Tabar, \emph{{Graphs determined by
  signless Laplacian spectra}}, AKCE Int. J. Graphs Combin. (2018),
  https://doi.org/10.10166/j.akcej.2018.06.009.

\bibitem{mul}
F.~Ayoobi, G.R. Omidi, and B.~Tayfeh-Rezaie, \emph{{A note on graphs whose
  signless Laplacian has three distinct eigenvalues}}, Linear and Multilinear
  Algebra \textbf{59} (2011), 701--706.

\bibitem{RBBbook}
R.B. Bapat, \emph{{Graphs and Matrices}}, Springer-Verlag, New York, 2010.

\bibitem{j3}
C.~Bu and J.~Zhou, \emph{{Signless Laplacian spectral characterization of the
  cones over some regular graphs}}, Linear Algebra Appl. \textbf{436} (2012),
  3634--3641.

\bibitem{bookc}
D.~Cvetkovi$\acute{\text{c}}$, P.~Rowlinson, and S.~Simi$\acute{\text{c}}$,
  \emph{{An Introduction to the Theory of Graph Spectra}}, Cambridge University
  Press, Cambridge, 2010.

\bibitem{a2}
D.~Cvetkovi$\acute{\text{c}}$, P.~Rowlinson, and S.~K. Simi$\acute{\text{c}}$,
  \emph{{Signless Laplacians of finite graphs}}, Linear Algebra Appl.
  \textbf{423} (2007), 155--171.

\bibitem{split}
K.~C. Das and M.~Liu, \emph{{Complete split graph determined by its (signless)
  Laplacian spectrum}}, Discrete Appl. Math. \textbf{205} (2016), 45--51.

\bibitem{kite}
\bysame, \emph{{Kite graphs determined by their spectra}}, Appl. Math. Comput.
  \textbf{297} (2017), 74--78.

\bibitem{kcdas}
K.C. Das, \emph{{ On conjectures involving second Largest signless Laplacian
  eigenvalue of graphs}}, Linear Algebra Appl. \textbf{432} (2010), 3018--3029.

\bibitem{algebraic}
N.M.M. de~Abreu, \emph{{Old and new results on algebraic connectivity of
  graphs}}, Linear Algebra Appl. \textbf{423} (2007), 53--73.

\bibitem{j7}
L.S. de~Lima and V.~Nikiforov, \emph{{On the second largest eigenvalue of the
  signless Laplacian}}, Linear Algebra Appl. \textbf{438} (2013), 1215--1222.

\bibitem{p1}
Y.-Z. Fan and D.~Yang, \emph{{The signless Laplacian spectral radius of graphs
  with given number of pendant vertices}}, Graphs Combin. \textbf{25} (2009),
  291--298.

\bibitem{siam}
R.~Grone and R.~Merris, \emph{{The Laplacian spectrum of a graph II
  $^{\ast}$}}, SIAM J. Discrete Math. \textbf{7} (1994), 221--229.

\bibitem{siam1}
R.~Grone, R.~Merris, and V.~S. Snder, \emph{The {L}aplacian spectrum of a
  graph}, SIAM J. Matrix Anal. Appl. \textbf{11} (1990), 218--238.

\bibitem{primas}
H.~G\"{u}nthard and H.~Primas, \emph{{Zusammenhang von Graphentheorie und
  MO-Theorie von Molekeln mit Systemen konjugierter Bindungen}}, Helv. Chim.
  Acta \textbf{39} (1956), 1645--1653.

\bibitem{mull}
J-M Guo, L.~Feng, J-M Zhang, and Dongying, \emph{{ On the multiplicity of
  Laplacian eigenvalues of graphs}}, Czechoslovak Math. J. \textbf{60} (2010),
  689--–698.

\bibitem{lst}
I.~Gutman, \emph{{ The star is the tree with greatest Laplacian eigenvalue}},
  Kragujevac J. Math. \textbf{24} (2002), 61--65.

\bibitem{R3}
C.~He and E.~R. van Dam, \emph{{Laplacian spectral characterization of roses}},
  Linear Algebra Appl. \textbf{536} (2018), 19--30.

\bibitem{book}
R.~A. Horn and C.~R. Johnson, \emph{{Matrix Analysis}}, Cambridge University
  Press, New York, 2012.

\bibitem{dsi}
S.~Huang, J.~Zhou, and C.~Bu, \emph{{ Signless Laplacian spectral
  characterization of graphs with isolated vertices}}, Filomat \textbf{30}
  (2017), 3689--3696.

\bibitem{U2}
J.~Li and W-C Shiu, \emph{{ Unicyclic graphs with a perfect matching having
  signless Laplacian eigenvalue two}}, J. Mathematical Research with
  Applications \textbf{37} (2017), 379--390.

\bibitem{dsla}
J-S Li and Y-L Pan, \emph{{A note on the second largest eigenvalue of the
  laplacian matrix of a graph}}, Linear and Multilinear Algebra \textbf{48}
  (2000), 117--121.

\bibitem{j2}
M.~Liu, \emph{{Some graphs determined by their (signless) Laplacian spectra}},
  Czechoslovak Math. \textbf{62} (2012), 1117--1134.

\bibitem{butter}
M.~Liu, Y.~Zhu, H.~Shan, and K.~C. Das, \emph{{The spectral characterization of
  butterfly-like graphs}}, Linear Algebra Appl. \textbf{513} (2017), 55--68.

\bibitem{j5}
X.~Liu and P.~Lu, \emph{{Signless Laplacian spectral characterization of some
  joins}}, Electron. J. Linear Algebra \textbf{30} (2014), 443--454.

\bibitem{LJ1}
X.~Liu and S.~Wang, \emph{{Laplacian spectral characterization of some graph
  products}}, Linear Algebra Appl. \textbf{437} (2012), 1749--1759.

\bibitem{dh}
P.~Lu and X.~Liu, \emph{{ Laplacian spectral characterization of some double
  starlike trees}}, arXiv preprint arXiv:1205.6027 (2012).

\bibitem{d4}
M.~Mirzakhah and D.~Kiani, \emph{{ The sun graph is determined by its signless
  Laplacian spectrum}}, Electron. J. Linear Algebra. \textbf{20} (2010),
  610--620.

\bibitem{LJ}
L.~Sun, , W.~Wang, J.~Zhou, and C.~Bu, \emph{{Laplacian spectral
  characterization of some graph join}}, Indian J. Pure Appl. Math. \textbf{46}
  (2015), 279--286.

\bibitem{R2}
H.~Topcu, S.~Sorgun, and W.~H. Haemers, \emph{{On the spectral characterization
  of pineapple graphs}}, Linear Algebra Appl. \textbf{507} (2016), 267--273.

\bibitem{Dam}
E.~R. van Dam and W.~H. Haemers, \emph{{ Which graphs are determined by their
  spectra}}, Linear Algebra Appl. \textbf{373} (2003), 241--272.

\bibitem{Dam1}
\bysame, \emph{{Developments on spectral characterizations of graphs}},
  Discrete Math. \textbf{309} (2009), 576--586.

\bibitem{j4}
L.~Xu and C.~He, \emph{{On the signless Laplacian spectral determination of the
  join of regular graphs}}, Discrete Math. Algorithm. Appl. \textbf{6} (2014),
  1450050.

\bibitem{js}
G.L. Yu, Y.R. Wu, and J.L. Shu, \emph{{ Signless Laplacian spectral radii of
  graphs with given chromatic number}}, Linear Algebra Appl. \textbf{435}
  (2011), 1813--1822.

\bibitem{LJ2}
J.~Zhou and C.~Bu, \emph{{Laplacian spectral characterization of some graphs
  obtained by product operation}}, Discrete Math. \textbf{312} (2012),
  1591--1595.

\end{thebibliography}
\end{document}